\documentclass[a4paper,11pt,reqno]{amsart}
\usepackage[utf8]{inputenc}
\usepackage[english]{babel}
\usepackage{amsmath}
\usepackage{amsfonts}
\usepackage{amsthm}
\usepackage{mathtools}
\usepackage{float}
\usepackage{stmaryrd}
\usepackage{graphics}
\usepackage{graphicx}
\usepackage{subfig}
\usepackage{enumitem}
\usepackage{datetime}
\usepackage{comment}
\usepackage{todonotes}
\usepackage{bm}
\usepackage[colorlinks=false]{hyperref}
\usepackage{color}

\newcommand{\mc}{\mathcal}

\newcommand{\eps}{\varepsilon}
\renewcommand{\d}{\,\mathrm{d}}
\newcommand{\tr}{\mathrm{tr}\,}

\DeclareMathOperator*{\sym}{sym}

\newcommand\R{\mathbb R}

\renewcommand{\to}{\rightarrow}

\def\Xint#1{\mathchoice
	{\XXint\displaystyle\textstyle{#1}}%
	{\XXint\textstyle\scriptstyle{#1}}%
	{\XXint\scriptstyle\scriptscriptstyle{#1}}%
	{\XXint\scriptscriptstyle\scriptscriptstyle{#1}}%
	\!\int}
\def\XXint#1#2#3{{\setbox0=\hbox{$#1{#2#3}{\int}$}
		\vcenter{\hbox{$#2#3$}}\kern-.5\wd0}}

\def\dashint{\Xint-}

\numberwithin{equation}{section}
\newtheorem{thm}{Theorem}[section]
\newtheorem{defn}[thm]{Definition}
\newtheorem{prop}[thm]{Proposition}
\newtheorem{lemma}[thm]{Lemma}

\theoremstyle{definition}
\begingroup
\newtheorem{rmk}[thm]{Remark}

\endgroup

\theoremstyle{remark}

\title[Von K\'arm\'an shells and pressure]{Von K\'arm\'an theories for thin elastic shells subject to pressure}

\author[T. H. Giang and F. Riva]{Trung Hieu Giang and Filippo Riva}

\address[T. H. Giang]{Charles University, Faculty of Mathematics and Physics, Department of Mathematical Analysis, Sokolovsk\'a 49/83, 186 75 Prague 8, Czechia; Vietnam Academy of Science and Technology, Institute of Mathematics, 18 Hoang Quoc Viet, Hanoi, Vietnam}
\email{trung-hieu.giang@matfyz.cuni.cz}

\address[F. Riva]{Charles University, Faculty of Mathematics and Physics, Department of Mathematical Analysis, Sokolovsk\'a 49/83, 186 75 Prague 8, Czechia; Institute of Information Theory and Automation, Czech Academy of Sciences, Pod vod\'arenskou ve\v z\'i 4, CZ-182 00, Prague 8, Czechia}
\email{filippo.riva@matfyz.cuni.cz}

\begin{document}

	\begin{abstract}
		The variational derivation of two-dimensional von K\'arm\'an shell energies from three-dimensional nonlinear elasticity is investigated in the presence of constant pressure loads. On the one hand, the paper extends the existing analysis of pressure-loaded plates to the curved setting of shells. On the other hand, it complements the classical variational derivation of shell theories by incorporating a nonlinear pressure force.
		
	\end{abstract}
	
	\maketitle
	
	{\small
		\par
	}
	
	\pagenumbering{arabic}
	
	\medskip
	
	\tableofcontents
	
	\section*{Introduction}
	In the framework of materials science, understanding the effective behavior of low-dimensional elastic structures starting from the underlying three-dimensional theory represents a fundamental issue, also due to its numerous applications in engineering. Over the past few decades, the rigorous analysis of dimension-reduction problems has attracted considerable attention within the mathematical community, particularly in the Calculus of Variations. A variational approach to dimension-reduction, based on the tool of $\Gamma$-convergence, was pioneered in \cite{LR95}, where the authors established the derivation of nonlinear membrane theory. This line of research culminated in the seminal works \cite{FJM02,FJM06}, which provided an entire hierarchy of two-dimensional models encompassing, in particular, the so-called bending and von K\'arm\'an regimes, by developing the celebrated geometric rigidity estimate.
	
	These results were subsequently extended to shells, namely thin elastic bodies whose reference configurations need not be flat, but may instead be neighborhoods of general surfaces embedded in $\R^3$; see, e.g., \cite{ciarlet3}. We mention \cite{LR96} for the membrane theory (see also \cite{Cosserat} for the Cosserat model), \cite{FJMM03} for the bending regime, and \cite{LMP10} for the von K\'arm\'an regimes, while we refer the interested reader to the monograph \cite{Lew13} for a comprehensive treatment of the subject. Besides elasticity, dimension-reduction for shells has also been developed in the context of plasticity \cite{MaggMora}, brittle fracture \cite{AlmiPerotto}, and magnetoelasticity \cite{Tasso}.
	
	It is worth noting, however, that all the aforementioned works focus primarily on the internal energy of the body, while external forces are restricted to those described solely by linear functionals. This choice corresponds to the modeling of so-called \emph{dead loads}, namely forces that are independent of deformations. Although dead loads can be incorporated into the mathematical analysis rather straightforwardly, their mechanical relevance is limited compared with the broader and more realistic class of forces that act on the deformed configuration of the body. Such forces are necessarily represented by nonlinear functionals, and are commonly referred to as \emph{live loads} \cite{ciarlet,PodGuid}.
	
	Interest in live loads within variational approaches to elasticity has recently been renewed. For pressure forces acting on the boundary of an elastic material, a rigorous linearization procedure was developed in \cite{MR23}, while dimension-reduction problems were investigated in \cite{KR25}, encompassing all the classical energy scalings, albeit only for plates. A notable outcome of this analysis is that, in contrast to dead loads, the nonlinear nature of pressure forces leads to nontrivial interactions with the bulk elastic energy in the limit, at least in the low-energy regime associated with von K\'arm\'an theories.    
	
	Motivated by these considerations, in this paper we propose to extend the analysis of \cite{KR25} to the setting of shells. We focus on the von K\'arm\'an regimes, thereby complementing the analysis of \cite{LMP10} by replacing dead loads with pressure forces. This restriction is essentially based on two reasons. First, as shown in \cite{KR25}, nontrivial interactions between pressure forces and the bulk elastic energy arise only under low-energy scalings. Second, in the bending regimes, the dependence of the limiting energy density on the position along the midsurface gives rise to substantial technical difficulties, particularly in the construction of recovery sequences, which are beyond the scope of the present analysis.
	
	The paper is organized as follows. In Section \ref{sec:preliminaries}, we introduce the notation used throughout the paper and we recall some useful properties and definitions regarding surfaces. Section \ref{sec:elasticenergy} describes the geometry of shells, introduces the corresponding three-dimensional energy including both elastic and pressure contributions, and lists the main assumptions required for the dimension-reduction analysis. Section \ref{sec:statement} presents our main results, concerning the convergence of minimizers of the three-dimensional model to minimizers of von K\'arm\'an-type functionals. Section \ref{sec:proof} is finally devoted to the proofs of these results, which rely on $\Gamma$-convergence arguments.

	\section{Notation and geometric preliminaries}
	\label{sec:preliminaries} 
	
	We denote by $\R^{d\times d}$, $\R^{d\times d}_{\rm sym}$ and  $\R^{d\times d}_{\rm skew}$ the set of $d\times d$ real matrices and its subsets of symmetric and skew-symmetric matrices, respectively. The space of linear and continuous operators from a normed vector space $X$ to itself is indicated by $\mc L(X)$. Given a matrix $F\in \R^{d\times d}$, its symmetric part $\frac{F+F^T}{2}$is denoted by $\operatorname{sym} F$, while we write $\operatorname{skew} F$ for its skew-symmetric part $\frac{F-F^T}{2}$. The symbol $\operatorname{Id}_d$ stands for the identity matrix of dimension $d$. The set of rotations is denoted by $SO(d)$, that is,
	\begin{equation*}
		SO(d)=\{R\in\R^{d\times d}: \ R^TR= \operatorname{Id}_d, \ \det R=1\}.
	\end{equation*}
	
	Given two matrices $F,G\in\R^{d\times d}$, we write $F:G$ for the Frobenius scalar product, i.e. $F:G:=\tr F^TG$, while the euclidean scalar product of two vectors $\xi$ and $\zeta$ in $\R^d$ is indicated by $\langle\xi,\zeta\rangle$. The tensor product $\xi\otimes\zeta$ instead defines the $d\times d$ matrix with entries $(\xi\otimes\zeta)_{ij}=\xi_i\zeta_j$. Throughout the paper we adopt standard notation for Lebesgue and Sobolev spaces.	
	
	Given a two-dimensional, compact, and oriented surface $S$ in $\mathbb{R}^3$ of class $\mathcal{C}^{1,1}$, we denote by $\mathbf{n}$ the unit normal vector field on $S$, and by $T_z S$ the tangent space of $S$ at the point $z\in S$. Often, we will identify $T_z S$ either with $\R^2$, or with a subspace of $\R^3$. The surface measure on $S$ is indicated by the symbol $\d\sigma$. We also introduce the normal projection $\varrho^S$ onto $S$, namely
	\begin{equation}\label{eq:projection}
		\varrho^S(z + t\mathbf{n}(z)) = z \quad \text{for all } z \in S, \, |t| \ll 1.
	\end{equation}
	
	For any vector field $w \in W^{1,2}(S; \mathbb{R}^3)$, we write $\nabla^Sw$ for its surface gradient, while we denote by $\partial_{\tau}w(z)\in \R^3$ its tangential derivative at $z \in S$ in direction $\tau \in T_zS$, namely $\partial_\tau w(z)=\nabla^S w(z)\tau$. Then, for $z\in S$, we consider the shape operator
	$\Pi(z) \colon T_zS \to T_zS \subset \mathbb{R}^3$, which satisfies
	$$
	\Pi(z)\tau = \partial_{\tau}\mathbf{n}(z) \quad \text{for all }  \tau \in T_zS.
	$$
	With a slight abuse of notation, we identify $\Pi(z)$ with a ${3 \times 3}$ matrix by setting $\Pi(z)\mathbf{n}(z):= 0$. Then, it is well known that $\Pi : S \to \mathbb{R}_{\text{sym}}^{3 \times 3}$ is a bounded and symmetric matrix field on $S$.
	
	For the purposes of the paper, it will be convenient to recall the following notation from \cite{LMP10}, see also \cite[Chapter 5]{Lew13}. For a vector field $w \in W^{1,2}(S; \mathbb{R}^3)$, we define its symmetric surface gradient $\operatorname{sym} \nabla^Sw \in L^2(S; \mathbb{R}^{2 \times 2})$ as the field of bilinear forms on $T_zS$, defined by
	\begin{equation*}
		\langle \operatorname{sym} \nabla^Sw(z)\tau, \eta \rangle := \frac{\langle \partial_\tau w(z), \eta \rangle + \langle \partial_\eta w(z), \tau \rangle}{2} \quad \text{for all } \tau, \eta \in T_zS.
	\end{equation*}
	
	Moreover, the following spaces of functions will be crucial for the presented analysis. We refer to \cite{S79}, \cite{GSP95} and \cite{SP89} for details.
	
	\begin{defn}
		\label{def1}
		The space $\mathcal{V}(S)$ of \emph{infinitesimal isometries} on $S$ is defined as
		\begin{equation*}
			\mathcal{V}(S) = \left\{ V \in W^{2,2}(S; \mathbb{R}^3) : \ \operatorname{sym} \nabla^S V = 0 \ \text{on} \ S \right\}.
		\end{equation*}
		Equivalently, $ V \in W^{2,2}(S; \mathbb{R}^3) $ belongs to $\mathcal{V}(S) $ if and only if there exists a skew-symmetric matrix field $ A \in W^{1,2}(S; \R^{3\times 3}_{\rm skew}) $ such that 
		\begin{equation*}
			\partial_\tau V(z) = A(z) \tau\quad\text{for }z\in S \text{ and }\tau \in T_zS.
		\end{equation*}
		In this case, the matrix field $A$ is uniquely determined by $V$.
	\end{defn}
	
	\begin{defn}
		\label{def2}
		The space of \emph{finite strains} $ \mathcal{B}(S) $ consists of the symmetric matrix fields of the form
		\begin{equation*}
			\mathcal{B}(S) = \left\{ L^2 - \lim_{h \to 0} \operatorname{sym} \nabla^S w^h: \ w^h \in W^{1,2}(S; \mathbb{R}^3) \right\}.
		\end{equation*}
	\end{defn}
	
	\section{Elastic shells and pressure loads}\label{sec:elasticenergy}
	
	Let $S$ be a two-dimensional surface in $\mathbb{R}^3$ of class $\mathcal{C}^{1,1}$. We assume that $S$ is compact, connected, oriented, and that its boundary $\partial S$ is the union of finitely many (possibly none) Lipschitz continuous curves. Consider a family $\{S^h\}_{h > 0}$ of shells with midsurface $S$, representing the reference configuration of a thin hyperelastic three-dimensional body, namely
	$$
	S^h = \{x = z + t\mathbf{n}(z)\in \R^3: \, z \in S, \, -h/2 < t < h/2\}, \quad h \le h_0.
	$$
	Above, $h_0$ is chosen small enough such that the projection $\varrho^S$ in \eqref{eq:projection} is well defined on $S^h$ for all $h\le h_0$. In particular, observe that one could choose $h_0$ so that
	\begin{equation}\label{eq:bdh0}
		\frac 12 \le \det \big( \operatorname{Id}_3 + t \Pi(z) \big) \le \frac 32 \quad \text{for }z\in S\text{ and }t\in\left(-\frac{h_0}{2},\frac{h_0}{2}\right).
	\end{equation}
	
	We are interested in the situation of a uniform constant pressure applied to the shell. Following \cite{KR25} (see also \cite{MR23}), such model can be variationally understood by means of the total energy functional per unit thickness $\mathcal E_h^{\alpha, \pi}: W^{1,2}(S^h; \mathbb{R}^3) \to (-\infty, +\infty]$ given by
	\begin{equation}\label{eq:3denergy}
		\mc E_h^{\alpha, \pi}(u) = 
		\begin{cases} \displaystyle
			\frac 1h\int_{S^h} W(\nabla u(x)) \d x + h^{\alpha-1} \pi \int_{S^h} \det \nabla u(x) \d x, & \text{if } u \in \text{dom } \mc E_h^{\alpha, \pi}, \\
			+\infty, & \text{otherwise},
		\end{cases}
	\end{equation}
	where the domain $\text{dom } \mc E_h^{\alpha, \pi}$ consists of all deformations $u \in W^{1,2}(S^h; \mathbb{R}^3)$ such that both $W(\nabla u)$ and the Jacobian $\det \nabla u$ are integrable over $S^h$. The first integral above represents the elastic bulk energy stored in the material, where $W: \mathbb{R}^{3 \times 3} \to [0, +\infty]$ is the elastic energy density. Instead, the second integral is the potential associated with a constant pressure load of intensity $h^\alpha \pi \in \mathbb{R}$ acting on the boundary of the deformed configuration, see for instance \cite{ciarlet} or \cite{PodGuid}. As mentioned in the Introduction, the present paper covers low-energy regimes, corresponding to $\alpha\ge 2$, thus modeling weak pressure loads.
	
	We assume standard properties on the density $W$:
	\begin{subequations}
		\begin{align}
			\label{frameindif}
			&\text{\textbullet} \, W(RF) = W(F) \textrm{ for all } R \in SO(3) \text{ and } F \in \mathbb{R}^{3 \times 3};  \\
			&\text{\textbullet} \, W(F) = 0 \iff F \in SO(3);  \\
			\label{coercivity}
			&\text{\textbullet} \, W(F) \geq c \operatorname{dist}(F; SO(3))^2 \text{ for all } F \in \mathbb{R}^{3 \times 3};  \\
			\label{continuity}
			&\text{\textbullet} \, W \textrm{ is of class } C^2 \textrm{ in a neighborhood of } SO(3).
		\end{align}        
	\end{subequations}
	As pointed out in \cite{MR23,KR25}, the coercivity condition \eqref{coercivity} yields
	\begin{equation}
		\label{compare}
		W(F) \geq c |\det F - 1|^2 \quad \text{whenever } |\det F - 1| \leq 1. 
	\end{equation}
	
	Similarly to \cite{KR25}, we also need the following additional assumptions in order to deal with the pressure term in the asymptotic analysis $h\to 0$. We require that either
	\begin{equation}
		\label{positivity}
		\pi > 0 \quad \text{and} \quad W(F) = +\infty \text{ if } \det F \leq 0, 
	\end{equation}
	or 
	\begin{equation}
		\label{linearcontrol}
		W(F) \geq c |\det F - 1| \quad \text{whenever } |\det F - 1| > 1.
	\end{equation}
	
	In dimension-reduction procedures, it is useful to properly reformulate the problem under study on a fixed domain, in this context given by the shell $S^{h_0}$. To this aim, given a deformation $u\in  W^{1,2}(S^h; \mathbb{R}^3)$, we consider the rescaled deformation $y\in  W^{1,2}(S^{h_0}; \mathbb{R}^3)$ defined as
	\begin{equation*}
		y(z + t\mathbf{n}(z)) := u\left(z + t \frac{h}{h_0} \mathbf{n}(z)\right)\quad\text{for }z\in S,\,|t|<\frac{h_0}{2}.
	\end{equation*}
	We also introduce the rescaled gradient $\nabla_h y$ as
	\begin{equation*}
		\nabla_h y(z + t\mathbf{n}(z)):= \nabla u\left(z + t \frac{h}{h_0} \mathbf{n}(z)\right) \quad \text{for } z \in S, \, |t| < \frac{h_0}{2}.
	\end{equation*}
	
	The following lemma states the relevant properties of such scaling and provides a useful change of variables formula for the energy \eqref{eq:3denergy}.  For more details, we refer to \cite[Lemma 5.5]{Lew13}. We stress that, with a slight abuse of notation, we still write $\mc E_h^{\alpha,\pi}$ for the rescaled energy with argument $y\in  W^{1,2}(S^{h_0}; \mathbb{R}^3)$.
	
	\begin{lemma}
		\label{lem:changeofvarables}
		For $z\in S$, $t\in \left(-\frac{h_0}{2},\frac{h_0}{2}\right)$, and $\tau \in T_zS$ there holds
		\begin{align*}
			\partial_{\tau} y(z + t \mathbf{n}(z)) &= \nabla_h y(z + t \mathbf{n}(z)) (\mathrm{Id}_3 + t \tfrac{h}{h_0} \Pi(z)) (\mathrm{Id}_3 + t \Pi(z))^{-1} \tau,\\
			\partial_{\mathbf{n}} y(z + t \mathbf{n}(z)) &= \tfrac{h}{h_0} \nabla_h y(z + t \mathbf{n}(z)) \mathbf{n}(z).
		\end{align*}    
		Moreover, for $y$ induced by $u \in \operatorname{dom} \mathcal E_h^{\alpha, \pi}$, we can write:
		\begin{align*}
			\mathcal{E}^{\alpha,\pi}_h(y) &= \frac{1}{h_0} \int_{S^{h_0}} W(\nabla_h y(z + t \mathbf{n}(z))) \det\big((\mathrm{Id}_3 + t \tfrac{h}{h_0} \Pi(z)) (\mathrm{Id}_3 + t \Pi(z))^{-1}\big) \d x \\
			& \quad + \frac{h^\alpha \pi}{h_0} \int_{S^{h_0}} \det \nabla_h y(z + t \mathbf{n}(z)) \det\big((\mathrm{Id}_3 + t \tfrac{h}{h_0} \Pi(z)) (\mathrm{Id}_3 + t \Pi(z))^{-1}\big) \d x\\
			&= \int_S \dashint_{-h_0/2}^{h_0/2} W(\nabla_h y(z + t \mathbf{n}(z))) \mathcal{D}_h(z,t) \d t  \d \sigma(z)\\
			& \quad + h^\alpha \pi\int_S \dashint_{-h_0/2}^{h_0/2} \det \nabla_h y(z + t \mathbf{n}(z)) {\mc D}_h(z,t) \d t \d\sigma(z),
		\end{align*}
		where
		\begin{equation}\label{eq:Dh}
			{\mc D}_h(z,t) := \det\bigg(\operatorname{Id}_3 + t\frac{h}{h_0} \Pi(z)\bigg) \quad\textrm{ for } z\in S\text{ and }|t|<\frac{h_0}{2}.
		\end{equation}
	\end{lemma}
	
	\begin{rmk}
		Observe that, due to the choice of $h_0$, the bound \eqref{eq:bdh0}  implies the control
		\begin{equation}\label{fixed-h0}
			\frac 12 \le {\mc D}_h(z,t) \le \frac 32 \quad \text{for }z\in S\text{ and }|t|<\frac{h_0}{2}.
		\end{equation}
	\end{rmk}
	
	\section{Statement of the main results}
	\label{sec:statement}
	
	In this section we state our main results, concerning the convergence of minimizers of the three-dimensional energy \eqref{eq:3denergy} to minimizers of effective two-dimensional energies, in the form of von K\'arm\'an-type functionals.     
	
	To describe the limit models arising through the dimension-reduction procedure, it is useful to introduce the following notation. We first consider the quadratic form $Q_3$ on $\mathbb{R}^{3\times 3}$ defined by
	\begin{equation*}
		Q_3(F) := \mathbb CF:F\quad \textrm{ for all } F \in \mathbb{R}^{3\times 3},
	\end{equation*}
	where $\mathbb C:=D^2 W(\mathrm{Id}_3)$ is the elasticity tensor, which we recall is symmetric. For $z\in S$, we then define the contraction of $Q_3$ on $T_zS$ as
	\begin{equation}\label{eq:quad2}
		Q_2(z, G):= \min \{{Q}_3(F ): F_{\rm tan}(z)=G \} \quad \text{for } G \in \mathcal L(T_z S),
	\end{equation}
	where for any $F\in \R^{3\times 3}$ we denote by $F_{\rm tan}(z)\in \mathcal L(T_z S)$ the unique linear operator from $T_z S$ to itself satisfying the identity
	\begin{equation}\label{eq:tan}
		\langle F_{\rm tan}(z)\tau, \eta \rangle = \langle F\tau, \eta \rangle \quad \text{for all } \tau, \eta \in T_zS.
	\end{equation}
	Conversely, given $G\in \mathcal L(T_z S)$, we extend it to a $3\times 3$ matrix $\overline G(z)$ by setting $\overline G(z)\mathbf{n}(z):=0$. In coordinates, if $(\tau_1(z),\tau_2(z),\mathbf{n}(z))$ denotes an orthonormal bases of $\R^3$, observe that one can write
	\begin{equation}\label{eq:bar}
		\overline G(z)=\sum_{i,j=1}^2\langle G\tau_j(z),\tau_i(z)\rangle\tau_i(z)\otimes\tau_j(z).
	\end{equation}    
	It is well known that both $Q_3$ and $Q_2(z,\cdot)$ depend only on the symmetric part of their argument, due to \eqref{frameindif}, and that they are positive-definite when restricted to the subspace of symmetric operators, due to \eqref{coercivity}. We refer for instance to \cite[Lemmas 5.2 and 5.7]{Lew13}.
	
	Notice that the request $F_{\rm tan}(z)=G$ in \eqref{eq:quad2} is equivalent to $(F-\overline G(z))_{\rm tan}(z)=0$. Thus, since the only $3\times 3$ matrices inducing the null form on $T_z S$ are of the form
	\begin{equation*}
		a_1\otimes\mathbf{n}(z)+\mathbf{n}(z)\otimes a_2,\quad\text{for }a_1,a_2\in \R^3,
	\end{equation*}
	recalling that $Q_3$ and $Q_2$ depend only on the symmetric part of their argument, we can write
	\begin{equation}\label{eq:quad2'}
		Q_2(z, G)= \min_{a\in\mathbb{R}^3} \{{Q}_3(\overline G(z)+a\otimes \mathbf{n}(z)) \} \quad \text{for } G \in \mathcal L(T_z S).
	\end{equation}
	Following this approach, for $z\in S$ we finally introduce the function
	\begin{equation}
		\label{Q2pi}
		Q_2^\pi (z, G) = \min_{a\in\mathbb{R}^3} \{Q_3 (\overline G(z) + a \otimes \mathbf{n}(z) ) + 2\pi \langle a, \mathbf{n}(z)\rangle\} \quad \text{for } G \in \mathcal L(T_z S),
	\end{equation}
	which encodes the interaction effects between the bulk elastic energy and the pressure load. The next lemma clarifies the dependence of $Q_2^\pi$ on the parameter $\pi$.
	\begin{lemma}
		For any $z\in S$, there exists a linear operator $\mathfrak L(z, \cdot) : T_z S \to \mathbb{R}$ and a real number $\kappa(z)$ such that 
		\begin{equation}
			\label{Q2pivalue}
			Q_2^\pi (z, G) = Q_2 (z, G) + \pi \mathfrak L(z, G) + \pi^2 \kappa(z)\quad\text{for all }G \in \mathcal L(T_z S).
		\end{equation}
		More precisely, we have
		\begin{subequations}
			\begin{equation}\label{eq:Q2}
				Q_2(z,G)={Q}_3(\overline G(z))-\langle \mathbb M^{-1}_{\mathbf{n}(z)}\mathbb C\overline G(z)\mathbf{n}(z),\mathbb C\overline G(z)\mathbf{n}(z)\rangle,
			\end{equation}
			\begin{equation}\label{eq:L}
				\mathfrak L(z,G)=-2\langle\mathbb M^{-1}_{\mathbf{n}(z)}\mathbf{n}(z),\mathbb C\overline G(z)\mathbf{n}(z)\rangle,
			\end{equation}
			and 
			\begin{equation}\label{eq:k}
				\mathfrak \kappa(z)=-\langle\mathbb M^{-1}_{\mathbf{n}(z)}\mathbf{n}(z),\mathbf{n}(z)\rangle,
			\end{equation}            
		\end{subequations}
		where $\mathbb M_{\mathbf{n}(z)}\in\R^{3\times 3}_{\rm sym}$ is the symmetric matrix characterized by
		\begin{equation}\label{eq:Mndef}
			\mathbb M_{\mathbf{n}(z)}a=\mathbb C(a\otimes \mathbf{n}(z))\mathbf{n}(z)\quad\text{ for all }a\in\R^3.
		\end{equation}
		
		In particular, since $S$ is of class $C^{1,1}$, the maps $z\mapsto \|\mathfrak L(z,\cdot)\|_*$ and $z\mapsto \kappa(z)$ are Lipschitz continuous on $S$.
	\end{lemma}
	\begin{proof}
		For the sake of brevity, throughot the proof we do not stress the dependence on $z$. Expanding the term $Q_3(\overline G+a\otimes \mathbf{n})$ in the definition of $Q_2^\pi$, with simple manipulations we deduce
		\begin{align*}
			Q_2^\pi(z,G)&=  \min_{a\in\mathbb{R}^3} \{{Q}_3(\overline G)+2\langle\mathbb C\overline G\mathbf{n},a\rangle+\langle\mathbb M_{\mathbf{n}}a,a\rangle+2\pi\langle\mathbf{n},a\rangle \}\\
			&={Q}_3(\overline G)+\min_{a\in\mathbb{R}^3} \{\langle\mathbb M_{\mathbf{n}}a,a\rangle+2\langle\mathbb C\overline G\mathbf{n}+\pi \mathbf n,a\rangle \}.
		\end{align*}
		Recalling that $\mathbb M_{\mathbf{n}}$ is positive-definite---hence invertible---since $\mathbb C$ is positive-definite on symmetric matrices, the minimum point $a^\pi$ can be computed as
		\begin{equation*}
			a^\pi=-\mathbb M^{-1}_{\mathbf{n}}( \mathbb C\overline G\mathbf{n}+\pi \mathbf n).
		\end{equation*}
		We thus infer
		\begin{align*}
			Q_2^\pi(z,G)&= {Q}_3(\overline G)-\langle  \mathbb C\overline G\mathbf{n}+\pi \mathbf n,\mathbb M^{-1}_{\mathbf{n}}( \mathbb C\overline G\mathbf{n}+\pi \mathbf n)\rangle\\
			&={Q}_3(\overline G)-\langle\mathbb C\overline G\mathbf{n},\mathbb M^{-1}_{\mathbf{n}}\mathbb C\overline G\mathbf{n}\rangle-2\pi\langle\mathbf n,\mathbb M^{-1}_{\mathbf{n}}\mathbb C\overline G\mathbf{n}\rangle-\pi^2\langle\mathbf n,\mathbb M^{-1}_{\mathbf{n}}\mathbf n\rangle,
		\end{align*}
		and we conclude since $\mathbb M_{\mathbf{n}}$ is symmetric.
	\end{proof}
	
	The effective two-dimensional theories will be described in terms of the following von K\'arm\'an-type energies, depending on infinitesimal isometries $V\in \mc V(S)$ and finite strains $B\in \mc B(S)$ of the surface $S$. If $\alpha=2$, the obtained limit functional $\mc E^\pi_{\rm vK}\colon \mathcal{V}(S)\times \mathcal{B}(S)\to \R$ takes the form
	\begin{subequations}\label{eq:envK}
		\begin{equation}\label{eq:envKpi}
			\mathcal{E}^\pi_{\rm vK}(V, B):=  \mathcal{E}_{\rm vK}(V, B)  +\pi\int_S\Bigg(\frac 12 \mathfrak L\left(z, B- \frac{1}{2}(A^2)_{\rm tan}\right)+ \operatorname{tr}_z \left( B - \frac{1}{2}(A^2)_{\rm tan}\right) \Bigg) \d \sigma(z),
		\end{equation}
		where the von K\'arm\'an energy
		$\mc E_{\rm vK}\colon \mathcal{V}(S) \times \mathcal{B}(S)\to [0,+\infty)$, introduced in \cite{LMP10},
		is given by
		\begin{equation}
			\mc E_{\rm vK}(V, B):= \frac 12 \int_S Q_2\left(z, B - \frac{1}{2}(A^2)_{\rm tan}\right) \d \sigma(z)+\frac{1}{24}\int_SQ_2(z, (\nabla^S(A \mathbf{n}) - A\Pi)_{\rm tan}) \d \sigma(z).
		\end{equation}              
	\end{subequations}
	In the two expressions above, $A$ is the matrix field associated with $V$ as in Definition \ref{def1}, $\mathfrak L$ is the linear operator introduced in \eqref{Q2pivalue}, and $\operatorname{tr}_z$ denotes the tangential trace on $T_z S$.
	
	If $\alpha>2$ instead, in the limit $h\to 0$ we derive the functional $\mc E^\pi_{\rm vK, lin}\colon \mathcal{V}(S)\times \mathcal{B}(S)\to \R$ defined by
	\begin{subequations}\label{eq:envKlin}
		\begin{equation}\label{eq:envKlinpi}
			\mathcal{E}^\pi_{\rm vK,lin}(V, B):= \mathcal{E}_{\rm vK,lin}(V, B)  +\pi\int_S\Bigg(\frac 12 \mathfrak L\left(z, B\right)  + \operatorname{tr}_z \left( B \right) \Bigg) \d \sigma(z),
		\end{equation}
		where the linearized von K\'arm\'an energy
		$\mc E_{\rm vK,lin}\colon \mathcal{V}(S) \times \mathcal{B}(S)\to [0,+\infty)$, again found in \cite{LMP10}, is given by
		\begin{equation}
			\mc E_{\rm vK,lin}(V, B):= \frac 12 \int_S Q_2\left(z, B \right) \d \sigma(z) +\frac{1}{24}\int_SQ_2(z, (\nabla^S(A \mathbf{n}) - A\Pi)_{\rm tan}) \d \sigma(z).
		\end{equation}           
	\end{subequations}
	Clearly, in the pressureless case $\pi=0$ we recover the functionals obtained in \cite{LMP10}.
	
	\begin{rmk}[\textbf{The isotropic case}]
		When the elastic body is isotropic, we can provide a more explicit expression for the limiting von K\'arm\'an energies, which better shows the interaction between elastic energy and pressure. Indeed, in this situation it is well-known (see \cite{ciarlet} or \cite{FJM02}) that
		\begin{equation*}
			Q_3(F)=2\mu|\sym F|^2+\lambda(\tr F)^2 \quad\text{for }F\in\R^{3\times 3},
		\end{equation*}
		where $\mu>0$ and $\lambda>-\frac 23\mu$ are the Lam\'e coefficients of the material. The two-dimensional functionals take then the form
		\begin{equation}\label{eq:finalformula}
			\mathcal{E}^\pi_{\rm vK}(V, B)=  \mathcal{E}_{\rm vK}(V, B)  +\frac{2\mu\pi}{2\mu+\lambda}\int_S \left(\operatorname{tr}_z  B + \frac{|A|^2-|A\mathbf{n}|^2}{2}\right) \d \sigma(z),
		\end{equation}
		and
		\begin{equation*}
			\mathcal{E}^\pi_{\rm vK,lin}(V, B)= \mathcal{E}_{\rm vK,lin}(V, B)  +\frac{2\mu\pi}{2\mu+\lambda}\int_S \operatorname{tr}_z \left( B \right)  \d \sigma(z).
		\end{equation*}
		
		To clarify the above expressions, it is enough to find an explicit formula for the linear operator $\mathfrak L$ in \eqref{eq:L}. To this aim, recalling that in the isotropic framework one has
		\begin{equation*}
			\mathbb C F=2\mu\sym F+\lambda(\tr F)\operatorname{Id}_3\quad\text{for }F\in\R^{3\times 3},
		\end{equation*}
		simple computations first show that
		\begin{equation}\label{eq:CGn}
			\mathbb C\overline G(z)\mathbf{n}(z)=\lambda (\tr \overline G(z))\mathbf{n}(z).
		\end{equation}
		Analogously, by means of \eqref{eq:Mndef}, for $a\in\R^3$ we can compute
		\begin{equation*}
			\mathbb M_{\mathbf{n}(z)}a=\mu a+(\mu+\lambda)\langle a,\mathbf{n}(z)\rangle \mathbf{n}(z) =   \mu a_{\rm tan}+(2\mu+\lambda)\langle a,\mathbf{n}(z)\rangle \mathbf{n}(z),  
		\end{equation*}
		where $a_{\rm tan}=a-\langle a,\mathbf{n}(z)\rangle \mathbf{n}(z)$, namely
		\begin{equation*}
			\mathbb M_{\mathbf{n}(z)}=\mu(\operatorname{Id}_3-\mathbf{n}(z)\otimes\mathbf{n}(z))+(2\mu+\lambda)\mathbf{n}(z)\otimes \mathbf{n}(z).
		\end{equation*}
		The inverse matrix takes then the form
		\begin{equation*}
			\mathbb M^{-1}_{\mathbf{n}(z)}=\frac{1}{\mu}(\operatorname{Id}_3-\mathbf{n}(z)\otimes\mathbf{n}(z))+\frac{1}{2\mu+\lambda}\mathbf{n}(z)\otimes \mathbf{n}(z),
		\end{equation*}
		whence
		\begin{equation}\label{eq:Mn}
			\mathbb M^{-1}_{\mathbf{n}(z)}\mathbf{n}(z)=\frac{1}{2\mu+\lambda}\mathbf{n}(z).
		\end{equation}
		Finally, combining \eqref{eq:CGn} and \eqref{eq:Mn} we directly obtain
		\begin{align*}
			\mathfrak L(z,G)&=-2\langle\mathbb M^{-1}_{\mathbf{n}(z)}\mathbf{n}(z),\mathbb C\overline G(z)\mathbf{n}(z)\rangle=-\frac{2\lambda}{2\mu+\lambda}\tr \overline G(z)\\
			&=-\frac{2\lambda}{2\mu+\lambda}\tr_z G,
		\end{align*}
		whence
		\begin{equation*}
			\pi\left(\frac{1}{2}\mathfrak L(z,G)+\tr_zG\right)=\frac{2\mu\pi}{2\mu+\lambda}\tr_z G.
		\end{equation*}
		Formula \eqref{eq:finalformula} now follows since, recalling that $A$ is skew-symmetric, there holds
		\begin{equation*}
			\tr_z(A^2)_{\rm tan}=\tr A^2-\langle A^2\mathbf n(z),\mathbf n(z)\rangle=-\tr A^TA-\langle A\mathbf n(z),A^T\mathbf n(z)\rangle=-|A|^2+|A\mathbf n(z)|^2.
		\end{equation*}       
		
		For the sake of completeness, using \eqref{eq:Q2} we also compute
		\begin{align*}
			Q_2(z,G)=&{Q}_3(\overline G(z))-\langle \mathbb M^{-1}_{\mathbf{n}(z)}\mathbb C\overline G(z)\mathbf{n}(z),\mathbb C\overline G(z)\mathbf{n}(z)\rangle\\
			=&2\mu|\sym \overline G(z)|^2+\frac{2\mu\lambda}{2\mu+\lambda}(\tr\overline G(z))^2\\
			=&2\mu|\sym G|_{T_zS}^2+\frac{2\mu\lambda}{2\mu+\lambda}(\tr_z G)^2,
		\end{align*}
		while by \eqref{eq:k} we infer
		\begin{equation*}
			\kappa(z)=-\frac{1}{2\mu+\lambda}.
		\end{equation*}            
	\end{rmk}
	
	Since the energy ${\mc E}^{\alpha,\pi}_h$ may fail to admit minimizers under solely the previous assumptions, it is costumary to employ the following notion of almost minimizers. We say that $y^h$ is an $\alpha$-minimizer of ${\mc E}^{\alpha,\pi}_h$ if
	\begin{equation}\label{eq:almostmin}
		\lim\limits_{h\to 0}\frac{1}{h^{2\alpha}}\Big({\mc E}^{\alpha,\pi}_h(y^h)-\inf {\mc E}^{\alpha,\pi}_h\Big)=0.
	\end{equation}
	We are now in a position to state our main results.
	
	\begin{thm}[\textbf{Von K\'arm\'an theory ($\boldsymbol{\alpha=2}$)}]
		\label{thm:vonkarman}
		Assume that $W$ satisfies \eqref{frameindif}-\eqref{continuity}, together with either \eqref{positivity} or \eqref{linearcontrol}, and let $y^h$ be a $2$-minimizer of ${\mc E}^{2,\pi}_h$. Then there exist constant rotations $\overline R _h\in SO(3)$ and constant vectors $c^h\in \R^3$ such that, introducing the normalized deformations
		\begin{equation}\label{eq:normdef}
			\widetilde{y}^h:=\overline R_h^T y^h-c_h,
		\end{equation}
		the following convergence properties as $h\to 0$ hold, up to a subsequence that we do not relabel: 
		\begin{itemize}
			\item[(i)] the sequence $\widetilde y^h$ converges strongly in $W^{1,2}(S^{h_0}; \mathbb{R}^3)$ to $\varrho^S$, where $\varrho^S$ is the projection onto $S$ introduced in \eqref{eq:projection};
			
			\item[(ii)] the scaled averaged displacements $V^h_2[\widetilde{y}^h]\colon S\to \R^3$, defined by
			\begin{equation}
				\label{eq:avdisp}
				V^h_2[\widetilde{y}^h] (z) := \frac{1}{h} \dashint_{-h_0/2}^{h_0/2}\left(\widetilde{y}^h(z + t \mathbf{n}(z)) - z \right) \d t,
			\end{equation}        
			converge strongly in $W^{1,2}(S;\mathbb{R}^3)$ to some infinitesimal isometry $V \in \mathcal{V}(S)$;
			
			\item[(iii)] the scaled strains $ \frac{1}{h} \sym \nabla^S V^h_2 [\widetilde{y}^h]$ converge weakly in $ L^2(S, \mathbb{R}^{2 \times 2}) $ to some finite strain $ B \in \mathcal{B}(S) $;
			
			\item[(iv)] the limit pair $(V,B)$ minimizes the functional $\mc E^\pi_{\rm vK}$ introduced in \eqref{eq:envKpi}.
			
		\end{itemize}
	\end{thm}
	
	\begin{thm}[\textbf{Linearized von K\'arm\'an theory ($\boldsymbol{\alpha>2}$)}]
		\label{thm:linvonkarman}
		Assume that $W$ satisfies \eqref{frameindif}-\eqref{continuity}, together with either \eqref{positivity} or \eqref{linearcontrol}, and let $y^h$ be an $\alpha$-minimizer of
		${\mc E}^{\alpha,\pi}_h$. Then there exist constant rotations
		$\overline R _h\in SO(3)$ and constant vectors $c^h\in \R^3$ such
		that the sequence of normalized deformations $\widetilde y^h$ defined in \eqref{eq:normdef} fulfils properties $(i)$, $(ii)$ and $(iii)$ of Theorem \ref{thm:vonkarman}, with $\alpha$-scaled averaged displacements
		\begin{equation}\label{eq:avdispalpha}
			V^h_\alpha[\widetilde{y}^h] (z) := \frac{1}{h^{\alpha-1}} \dashint_{-h_0/2}^{h_0/2}\left( \widetilde{y}^h(z + t \mathbf{n}(z)) - z \right) \d t,
		\end{equation}
		in place of \eqref{eq:avdisp}. 
		
		Moreover, the limit pair $(V,B)$ obtained in $(ii)$ and $(iii)$ minimizes the functional $\mc E^\pi_{\rm vK,lin}$ introduced in \eqref{eq:envKlinpi}.
	\end{thm}
	
	The proof of the above two theorems will be based on the well-consolidated tool of $\Gamma$-convergence \cite{Bra02,DM}. The stated convergence of (almost) minimizers is indeed a standard consequence of the following results, written in terms of the functionals
	\begin{equation}\label{eq:barE}
		\overline{\mc E}^{\alpha,\pi}_h(y):=\mc E^{\alpha,\pi}_h(y)-h^\alpha \pi\int_S \dashint_{-h_0/2}^{h_0/2} {\mc D}_h(z,t) \d t \d\sigma(z),
	\end{equation}
	which come in useful to identify the correct scaling of minimizing sequences. In particular, observe that this normalization does not affect minimizers or almost minimizers, in the sense of \eqref{eq:almostmin}, of the original energy $\mc E^{\alpha,\pi}_h$. 
	
	\begin{prop}[\textbf{Compactness}]\label{prop:compact}
		Assume that $W$ satisfies \eqref{frameindif}-\eqref{continuity}. Let $\alpha \geq 2$, and let $y^h$ such that
		$\overline{\mc E}_h^{\alpha,\pi}(y^h)\le C h^{2\alpha}$. Then the following properties hold:
		\begin{itemize}
			\item[(i)] If $\pi>0$, that is if \eqref{positivity} is in force, then
			\begin{equation}\label{eq:boundsdet}
				\begin{aligned}
					\int_{S} \int_{-h_0/2}^{h_0/2} &W(\nabla_h y^h(z+t\mathbf{n}(z))) \mathcal{D}_h(z,t) \d t\d \sigma(z) \\
					&+\int_{\{|\det\nabla_h y^h-1|\le 1\}}|\det\nabla_h y^h(x)-1|^2  \d x \le C h^{2\alpha},
				\end{aligned}           
			\end{equation}
			where $\mathcal{D}_h(z,t)$ has been introduced in \eqref{eq:Dh}.
			
			If instead $\pi<0$, so that
			\eqref{linearcontrol} holds, then, in addition to
			\eqref{eq:boundsdet}, one also has
			\begin{equation}\label{eq:boundsdetbig}
				\int_{\{|\det\nabla_h y^h-1|> 1\}}|\det\nabla_h y^h(x)-1| \d x\le C h^{2\alpha}.
			\end{equation}
			
			\item[(ii)] There exist rotation fields $R_h\in W^{1,2}(S;SO(3))$, constant
			rotations $\overline R _h\in SO(3)$, and constant vectors
			$c_h\in \R^3$ such that the normalized deformations $\widetilde y^h$ in \eqref{eq:normdef}, the $\alpha$-scaled averaged displacements $V^h_\alpha[\widetilde y^h]$ in \eqref{eq:avdispalpha}, the scaled strains $ \frac{1}{h} \sym \nabla^S V^h_\alpha [\widetilde{y}^h]$, and the scaled nonlinear strains defined as
			\begin{equation}\label{eq:strain}
				G_h(z+t\mathbf{n}(z)):=\frac{R_h(z)^T\nabla_h \widetilde{y}^h(z+t\mathbf{n}(z))-\operatorname{Id}_3}{h^\alpha}\quad\text{for }z\in S,|t|<\frac{h_0}{2},
			\end{equation}
			fulfill the following convergences, up to a subsequence that we do not relabel:
			\begin{subequations}\label{eq:convuv}
				\begin{align}
					&\widetilde y^h\xrightarrow[h\to 0]{} \varrho^S,&&\text{strongly in }W^{1,2}(S^{h_0};\R^{3}),\label{eq:convuva}\\
					&V^h_\alpha[\widetilde{y}^h]\xrightarrow[h\to 0]{} V,&&\text{strongly in }W^{1,2}(S;\R^3), \label{eq:convuvb}\\
					&\frac{1}{h} \sym \nabla^S V^h_\alpha [\widetilde{y}^h]\xrightharpoonup[h\to 0]{} B,&&\text{weakly in }L^{2}(S;\R^{2\times 2}), \label{eq:convuvc}\\
					&G_h\xrightharpoonup[h\to 0]{} G,&&\text{weakly in }L^{2}(S^{h_0};\R^{3\times 3}).\label{eq:Ghconv2}
				\end{align}
			\end{subequations}
			Here, $\varrho^S$ is the projection defined in \eqref{eq:projection}, while the limit objects satisfy $V \in \mathcal{V}(S)$,
			\begin{equation}\label{eq:Gv}
				G(z+t\mathbf{n}(z)) \tau= G_0(z)\tau+ \frac{t}{h_0}\left(\nabla^S(A\mathbf{n})(z)-A(z)\Pi(z)\right)\tau                
				\quad\text{for }z\in S,|t|<\frac{h_0}{2}, \tau \in T_z  S,
			\end{equation}
			where
			\begin{equation*}
				G_0(z) :=  \dashint_{-h_0/2}^{h_0/2} G(z+t\mathbf{n}(z)) \d t,
			\end{equation*}
			and $A \in W^{1,2}(S;\R^{3\times 3}_{\rm skew})$ comes from Definition \ref{def1}, while
			\begin{equation}\label{eq:B}
				B(z)=\begin{cases}
					\frac 12(A^2)_{\rm tan}(z)+\sym (G_0)_{\rm tan}(z),&\text{if }\alpha=2,\\
					\sym (G_0)_{\rm tan}(z),&\text{if }\alpha>2.
				\end{cases}
			\end{equation}
			
		\end{itemize}
	\end{prop}
	
	Before stating the $\Gamma$-convergence result, for the sake of brevity we now define the functional
	$\mc E^{\alpha,\pi}: \mathcal{V}(S) \times \mathcal{B}(S) \to \R$ by setting
	\begin{equation}\label{eq:enapi}
		\mc E^{\alpha,\pi}(V,B):=\begin{cases}\displaystyle
			\displaystyle\mc E^\pi_{\rm vK}(V,B)+\frac{\pi^2}{2}\int_S \kappa \d \sigma(z),& \text{if }\alpha=2,\\
			\displaystyle\mc E^\pi_{\rm vK, lin}(V,B)+\frac{\pi^2}{2}\int_S \kappa \d \sigma(z),& \text{if }\alpha>2.
		\end{cases}
	\end{equation}
	Above, the function $\kappa$ has been introduced in \eqref{Q2pivalue}, while the von K\'arm\'an-type functionals with pressure $\mc E^\pi_{\rm vK}$ and $\mc E^\pi_{\rm vK,lin}$ have been introduced in \eqref{eq:envKpi} and \eqref{eq:envKlinpi}, respectively.
	
	\begin{prop}[\textbf{$\Gamma$-convergence}]
		\label{prop:gammaconvergence}
		Assume that $W$ satisfies \eqref{frameindif}-\eqref{continuity}, together with either \eqref{positivity} or \eqref{linearcontrol}. Then the following two properties hold for any $\alpha\ge 2$:
		\begin{itemize}
			\item for every sequence $\{y^h\}_h\subseteq W^{1,2}(S^{h_0};\R^3)$
			such that $\overline{\mc E}^{\alpha,\pi}_h(y^h)\le Ch^{2\alpha}$, there exist rotation fields $R_h\in W^{1,2}(S;SO(3))$, constant rotations $\overline R _h\in SO(3)$, and constant vectors
			$c_h\in \R^3$ for which the convergences stated in \eqref{eq:convuv} hold, and
			\begin{equation}\label{eq:Gliminf}
				\mc E^{\alpha,\pi}(V,B)\le \liminf\limits_{h\to 0}\frac{1}{h^{2\alpha}}\overline{\mc E}^{\alpha,\pi}_h(y^h);
			\end{equation}
			
			\item for every
			$(V,B)\in \mathcal{V}(S) \times \mathcal{B}(S)$, there
			exist a sequence $\{y^h\}_h\subseteq W^{1,2}(S^{h_0};\R^3)$, rotation fields $R_h\in W^{1,2}(S;SO(3))$, constant rotations $\overline R _h\in SO(3)$, and constant vectors
			$c_h\in \R^3$such that the convergences stated in \eqref{eq:convuva}, \eqref{eq:convuvb}, and \eqref{eq:convuvc} hold, and
			\begin{equation}\label{eq:Glimsup}
				\limsup\limits_{h\to 0}\frac{1}{h^{2\alpha}}\overline{\mc E}^{\alpha,\pi}_h(y^h)\le  \mc E^{\alpha,\pi}(V,B).
			\end{equation}
		\end{itemize}
	\end{prop}

	\section{Compactness and $\Gamma$-convergence analysis}
	\label{sec:proof}
	
	This section is devoted to the proofs of the compactness and the $\Gamma$-convergence results. We stress that Theorems \ref{thm:vonkarman} and \ref{thm:linvonkarman} then follow by means of standard general results in $\Gamma$-convergence \cite{DM}. We begin with the compactness property.
	\begin{proof}[Proof of Proposition \ref{prop:compact}]
		All the statements contained in $(ii)$ are a consequence of \cite[Theorems 5.23, 5.25 and 5.26]{Lew13} (see also \cite{LMP10}), once we prove that
		\begin{equation*}
			\int_{S} \int_{-h_0/2}^{h_0/2} W(\nabla_h y^h(z+t\mathbf{n}(z))) \mathcal{D}_h(z,t) \d t\d \sigma(z)  \le C h^{2\alpha}.
		\end{equation*}
		Since the above estimate is a byproduct of \eqref{eq:boundsdet}, in order to conclude it is enough to show the validity of $(i)$.
		
		To this aim, we first consider the case $\pi>0$. For the sake of clarity, we define the set
		\begin{equation}\label{eq:setdet}
			\Omega_h^-:=\left\{x=z+t\mathbf{n}(z) \in S^{h_0}: \, \left|\det\nabla_h y^h (x)-1\right|\le 1\right\}.
		\end{equation}
		From the definition \eqref{eq:barE} of $\overline{\mc E}_h^{\alpha,\pi}$, recalling the change of variables formula in Lemma \ref{lem:changeofvarables} we notice that
		\begin{equation}\label{eq:firstest}
			\begin{aligned}
				&\int_{S}\int_{-h_0/2}^{h_0/2} W(\nabla_h y^h(z+t\mathbf{n}(z)))\mathcal{D}_h(z,t)\d t \d\sigma(z)\\
				\le & C h^{2\alpha}+h^\alpha\pi\int_{S^{h_0} }\Big(1-\det\nabla_h y^h(x)\Big)\frac{\mathcal{D}_h(z,t)}{\det(\operatorname{Id}_3+t\Pi(z))}\d x.
			\end{aligned}      
		\end{equation}
		By means of \eqref{eq:bdh0} and \eqref{fixed-h0}, since $\pi$ is positive we can continue the above inequality, obtaining
		\begin{equation}\label{eq:1est}
			\begin{aligned}
				\int_{S}\int_{-h_0/2}^{h_0/2} W(\nabla_h y^h(z+t\mathbf{n}(z)))\mathcal{D}_h(z,t)\d t \d\sigma(z)
				\le& Ch^{2\alpha}+ Ch^\alpha\int_{ \Omega_h^-}|\det\nabla_h y^h-1| \d x \\
				\le& Ch^{2\alpha}+ Ch^\alpha \|\det\nabla_h y^h-1\|_{L^2(\Omega_h^-)}.
			\end{aligned}   
		\end{equation}
		Therefore, by using \eqref{compare}, we infer
		\begin{equation*}
			\|\det\nabla_h y^h-1\|_{L^2(\Omega_h^-)}^2\le Ch^{2\alpha}+ Ch^\alpha \|\det\nabla_h y^h-1\|_{L^2(\Omega_h^-)},
		\end{equation*}
		so that
		\begin{equation}\label{eq:goal2}
			\|\det\nabla_h y^h-1\|_{L^2(\Omega_h^-)}^2\le Ch^{2\alpha}.
		\end{equation}
		Combining \eqref{eq:goal2} with \eqref{eq:1est}, we finally deduce \eqref{eq:boundsdet}.
		
		We now consider the case $\pi<0$, in which \eqref{linearcontrol} holds. From \eqref{eq:firstest} we first deduce
		\begin{equation}\label{eq:goal'}
			\begin{aligned}
				&\int_{S}\int_{-h_0/2}^{h_0/2} W(\nabla_h y^h(z+t\mathbf{n}(z)))\mathcal{D}_h(z,t)\d t \d\sigma(z)\\
				\le &  Ch^{2\alpha}+ Ch^\alpha \|\det\nabla_h y^h-1\|_{L^2(\Omega_h^-)}+Ch^{\alpha}\|\det\nabla_h y^h-1\|_{L^1(S^{h_0}\setminus \Omega_h^-)},
			\end{aligned}	
		\end{equation}
		whence, by means of \eqref{compare} and \eqref{linearcontrol}, we infer
		\begin{align*}
			&\|\det\nabla_h y^h-1\|_{L^2(\Omega_h^-)}^2+\|\det\nabla_h y^h-1\|_{L^1(S^{h_0}\setminus \Omega_h^-)}\\
			\le& Ch^{2\alpha}+ Ch^\alpha \|\det\nabla_h y^h-1\|_{L^2(\Omega_h^-)}+Ch^{\alpha}\|\det\nabla_h y^h-1\|_{L^1(S^{h_0}\setminus \Omega_h^-)}.
		\end{align*}	
		Simple manipulations of the above inequality finally yield
		\begin{equation}\label{eq:newbound}
			\|\det\nabla_h y^h-1\|_{L^2(\Omega_h^-)}^2+\|\det\nabla_h y^h-1\|_{L^1(S^{h_0}\setminus \Omega_h^-)} \le Ch^{2\alpha}.
		\end{equation}
		Estimates \eqref{eq:boundsdet} and \eqref{eq:boundsdetbig} are now a byproduct of \eqref{eq:goal'} and \eqref{eq:newbound}, and we conclude.
	\end{proof}
	
	We now prove the $\Gamma$-convergence result stated in Proposition \ref{prop:gammaconvergence}. 
	
	\begin{proof}[Proof of Proposition \ref{prop:gammaconvergence}]
		We provide the proof in the situation $\pi>0$. The case of negative pressure, where \eqref{linearcontrol} is in force, can be treated analogously, exploiting in addition \eqref{eq:boundsdetbig}.\medskip
		
		\noindent{\underline{\textbf{Liminf inequality.}}} We first prove the $\lim\inf$ inequality \eqref{eq:Gliminf}. Since $\overline{\mc E}^{\alpha,\pi}_h(y^h)\le Ch^{2\alpha}$, we know that $\det\nabla_h y^h$ is positive almost everywhere in $S^{h_0}$ and Proposition \ref{prop:compact} applies. Moreover, by \eqref{eq:normdef} and \eqref{eq:strain}, we may write 
		\begin{equation}\label{eq:yG}
			\nabla_h y^h(z+t\mathbf{n}(z))=\overline R_h R_h(z)(\operatorname{Id}_3+h^\alpha G_h(z+t\mathbf{n}(z))).
		\end{equation} 
		For the sake of clarity, we now introduce the following notation. Given a measurable set $A\subseteq S^{h_0}$ and a point $z\in \varrho^S(A)$, where we recall that $\varrho^S$ is the normal projection onto $S$, see \eqref{eq:projection}, we denote by $I_z(A)$ the normal fiber of $A$ at $z$, namely
		\begin{equation*}
			I_z(A)=\left\{t\in\left(-\frac{h_0}{2},\frac{h_0}{2}\right):\,z+t\mathbf n(z)\in A\right\}.
		\end{equation*}    
		By the very definition \eqref{eq:setdet} of $\Omega^-_h$, employing \eqref{eq:yG} and \eqref{frameindif} we infer
		\begin{align}
			\nonumber\frac{1}{h^{2\alpha}}\overline{\mc E}^{\alpha,\pi}_h(y^h)&= \int_{S}\dashint_{-h_0/2}^{h_0/2} \!\!\left(\! \frac{1}{h^{2\alpha}}W(\nabla_h y^h(z{+}t\mathbf{n}(z))) {+} \frac{\pi}{h^\alpha}(\det\nabla_h y^h(z{+}t\mathbf{n}(z)){-}1)\! \right)\mathcal{D}_h(z,t)\d t \d \sigma(z)\\
			&\ge \frac{1}{h^{2\alpha}}\int_{S}\dashint_{-h_0/2}^{h_0/2} W(\operatorname{Id}_3+h^\alpha G_h(z+t\mathbf{n}(z)))\mathcal{D}_h(z,t)\d t \d\sigma(z) \label{eq:comp1} \\
			& \nonumber\quad+\frac{\pi}{h^\alpha}\int_{\varrho^S(\Omega^-_h)}\frac{1}{h_0}\int_{I_z(\Omega_h^-)}(\det\nabla_h y^h(z+t\mathbf{n}(z))-1)\mathcal{D}_h(z,t)\d t \d \sigma(z).
		\end{align}
		Arguing as in \cite[Proof of Theorem 5.25(ii)]{Lew13} and using \eqref{eq:Ghconv2}, we obtain
		\begin{equation}\label{eq:comp2}
			\begin{aligned}
				&\liminf\limits_{h\to 0}\frac{1}{h^{2\alpha}}\int_{S}\dashint_{-h_0/2}^{h_0/2} W(\operatorname{Id}_3+h^\alpha G_h(z+t\mathbf{n}(z))) \mathcal{D}_h(z,t)\d t \d \sigma(z)\\
				&\ge \frac 12 \int_{S}\dashint_{-h_0/2}^{h_0/2} Q_3(G(z+t\mathbf{n}(z))) \d t \d\sigma(z).
			\end{aligned}
		\end{equation}
		
		To treat the second term in the last line of \eqref{eq:comp1}, we introduce the set
		\begin{equation}\label{eq:Ba}
			B_h^\alpha:=\{x\in S^{h_0}: h^{\alpha/2}|G_h(x)|\le 1\},
		\end{equation}
		and observe that $\lim\limits_{h\to 0}|S^{h_0}\setminus B^\alpha_h|=0$ by Chebyshev's inequality. Using again \eqref{eq:yG}, we split the integral under consideration as follows:
		\begin{align*}
			&\frac{1}{h^\alpha}\int_{\varrho^S(\Omega^-_h)}\int_{I_z(\Omega_h^-)}(\det\nabla_h y^h(z+t\mathbf{n}(z))-1)\mathcal{D}_h (z,t)\d t \d \sigma(z) \\
			&= \frac{1}{h^\alpha}\int_{\varrho^S(\Omega_h^- \cap B^\alpha_h)}\int_{I_z(\Omega_h^- \cap B^\alpha_h)}(\det(\operatorname{Id}_3+h^\alpha G_h(z+t\mathbf{n}(z))-1)\mathcal{D}_h(z,t)\d t \d \sigma(z) \\
			& \quad+\frac{1}{h^\alpha}\int_{\varrho^S(\Omega_h^- \setminus B^\alpha_h)}\int_{I_z(\Omega_h^- \setminus B^\alpha_h)}(\det\nabla_hy^h(z+t\mathbf{n}(z))-1)\mathcal{D}_h(z,t)\d t \d \sigma(z).
		\end{align*}
		By \eqref{eq:boundsdet}, \eqref{eq:bdh0}, and \eqref{fixed-h0}, we deal with the second integral above using the estimate:
		\begin{align*}
			&\frac{1}{h^\alpha}\int_{\varrho^S(\Omega_h^- \setminus B^\alpha_h)}\int_{I_z(\Omega_h^- \setminus B^\alpha_h)}|\det\nabla_h y^h(z+t\mathbf{n}(z))-1|\mathcal{D}_h(z,t) \d t \d\sigma(z) \\
			&\le \frac{C}{h^\alpha} \|\det\nabla_h y^h{-}1\|_{L^2(\Omega_h^-)}|S^{h_0}\setminus B_h^\alpha|^{1/2}\le C|S^{h_0}\setminus B_h^\alpha|^{1/2}\xrightarrow[h\to 0]{}0.
		\end{align*}
		For the first integral instead, we exploit the well-known determinant expansion
		\begin{equation}\label{eq:sviluppodet}
			\det(\operatorname{Id}_3+t F)=1+t\tr F+t^2\iota_2(F)+t^3\det F,\quad\text{for }F\in \R^{3\times 3},\,t\in\R,
		\end{equation}
		where we recall that $\iota_2$ satisfies $|\iota_2(F)|\le C|F|^2$. We can thus write
		\begin{align*}
			& \frac{1}{h^\alpha}\int_{\varrho^S(\Omega_h^- \cap B^\alpha_h)}\int_{I_z(\Omega_h^- \cap B^\alpha_h)}(\det(\operatorname{Id}_3+h^\alpha G_h(z+t\mathbf{n}(z))-1)\mathcal{D}_h(z,t)\d t \d \sigma(z)\\
			=& \int_{\varrho^S(\Omega_h^- \cap B^\alpha_h)}\int_{I_z(\Omega_h^- \cap B^\alpha_h)} \left[\tr G_h+ h^\alpha\iota_2(G_h)+ h^{2\alpha} \det G_h \right]\mathcal{D}_h(z,t)\d t \d \sigma(z).
		\end{align*}
		Using \eqref{eq:Ghconv2}, \eqref{eq:bdh0}, \eqref{fixed-h0} and the definition \eqref{eq:Ba} of $B_h^\alpha$, we now obtain
		\begin{align*}
			\bullet\, &  \int_{\varrho^S(\Omega_h^- \cap B^\alpha_h)}\int_{I_z(\Omega_h^- \cap B^\alpha_h)}h^\alpha|\iota_2(G_h)| \mathcal{D}_h(z,t)\d t \d \sigma(z)\le C h^\alpha\|G_h\|_2^2\le C h^\alpha \xrightarrow[h\to 0]{}0;\\
			\bullet\, &  \int_{\varrho^S(\Omega_h^- \cap B^\alpha_h)}\int_{I_z(\Omega_h^- \cap B^\alpha_h)}h^{2\alpha}|\det G_h| \mathcal{D}_h(z,t)\d t \d\sigma(z)\le Ch^{2\alpha}\int_{\Omega_h^- \cap B^\alpha_h}|G_h|^3\d x \\
			&  \quad\le C h^{3\alpha/2} \|G_h\|_2^2\le C h^{3\alpha/2} \xrightarrow[h\to 0]{}0;\\
			\bullet\, &\lim\limits_{h\to 0} \int_{\varrho^S(\Omega_h^- \cap B^\alpha_h)}\int_{I_z(\Omega_h^- \cap B^\alpha_h)}\tr G_h \mathcal{D}_h(z,t)\d t \d \sigma(z)=\int_{S}\int_{-h_0/2}^{h_0/2} \tr G\d t \d\sigma(z).
		\end{align*}
		For the last limit, besides \eqref{eq:Ghconv2}, we exploited the fact that $\mathcal{D}_h \to 1$ pointwise and that $|S^{h_0} \setminus (\Omega_h^- \cap B_h^\alpha)|$ vanishes as $h\to 0$. Indeed, we already know that $\lim\limits_{h\to 0}|S^{h_0}\setminus B_h^\alpha|=0$, and using \eqref{eq:sviluppodet} it is immediate to check that $B_h^\alpha \subseteq \Omega_h^-$ for $h$ sufficiently small. Therefore, we infer
		\begin{equation}\label{eq:comp3}
			\lim\limits_{h\to 0} \frac{1}{h^\alpha}\int_{\varrho^S(\Omega^-_h)}\int_{I_z(\Omega_h^-)}(\det\nabla_h y^h(z+t\mathbf{n}(z))-1)\mathcal{D}_h (z,t)\d t \d \sigma(z)= \int_{S}\int_{-h_0/2}^{h_0/2}  \tr G\d t\d\sigma(z).
		\end{equation}		
		Combining \eqref{eq:comp1}, \eqref{eq:comp2}, and \eqref{eq:comp3}, we thus obtain
		\begin{equation}\label{eq:a}
			\liminf\limits_{h\to 0}\frac{1}{h^{2\alpha}}\overline{\mc E}^{\alpha,\pi}_h(y_h)\ge \frac 12 \int_{S}\dashint_{-h_0/2}^{h_0/2} \Big(Q_3(G(z+t\mathbf{n}(z)))+2\pi\tr G(z+t\mathbf{n}(z))\Big)\d t \d \sigma(z).
		\end{equation}
		Since $G$ and $\overline{G_{\rm tan}}$ clearly induce the same bilinear form on the tangent space (see \eqref{eq:tan} and \eqref{eq:bar}), there exists a vector field $a\in L^2(S^{h_0};\R^3)$ such that
		\begin{equation*}
			\sym G(x)=\sym (\overline{G(x)_{\rm tan}}+a(x)\otimes \mathbf{n}(z)).
		\end{equation*}
		Recalling that $Q_3$ just depends on the symmetric part of its argument, and using the definition \eqref{Q2pi} of $Q_2^\pi$, we deduce
		\begin{equation}\label{eq:b}
			\begin{aligned}
				&\frac 12 \int_{S}\dashint_{-h_0/2}^{h_0/2} \Big(Q_3(G)+2\pi\tr G\Big)\d t \d \sigma(z)\\
				=&\frac 12 \int_{S}\dashint_{-h_0/2}^{h_0/2} \Big(Q_3(\overline{G_{\rm tan}}+a\otimes \mathbf{n})+2\pi\langle a,\mathbf{n}\rangle+2\pi\tr\overline{G_{\rm tan}} \Big)\d t \d \sigma(z)\\
				\ge & \frac 12 \int_{S}\dashint_{-h_0/2}^{h_0/2} \Big(Q_2^\pi(z,{G_{\rm tan}})+2\pi\tr_z{G_{\rm tan}} \Big)\d t \d \sigma(z).
			\end{aligned}      
		\end{equation}
		Since by \eqref{eq:Gv} we can characterize $G_{\rm tan}$ via the identity
		\begin{equation*}
			G_{\rm tan}(z+t\mathbf n(z))=(G_0)_{\rm tan}(z)+\frac{t}{h_0}(\nabla^S(A\mathbf{n})(z)-A(z)\Pi(z))_{\rm tan},
		\end{equation*}
		recalling \eqref{Q2pivalue} we can write
		\begin{equation}\label{eq:c}
			\begin{aligned}
				& \frac 12 \int_{S}\dashint_{-h_0/2}^{h_0/2} \Big(Q_2^\pi(z,{G_{\rm tan}})+2\pi\tr_z{G_{\rm tan}} \Big)\d t \d \sigma(z)\\
				=&\frac 12 \int_{S}\dashint_{-h_0/2}^{h_0/2} \Big(Q_2(z,(G_0)_{\rm tan})+\frac{t^2}{h_0^2}Q_2(z,(\nabla^S(A\mathbf{n})-A\Pi)_{\rm tan})\\
				&\qquad\qquad\qquad\qquad+\pi\mathfrak L(z, (G_0)_{\rm tan})+2\pi\tr_z{(G_0)_{\rm tan}}+\pi^2 \kappa  \Big)\d t \d \sigma(z)\\
				=&\frac 12 \int_{S}\Big(Q_2(z,(G_0)_{\rm tan})+\frac{1}{12}Q_2(z,(\nabla^S(A\mathbf{n})-A\Pi)_{\rm tan})\\
				&\qquad\qquad\qquad\qquad+\pi\mathfrak L(z, (G_0)_{\rm tan})+2\pi\tr_z{(G_0)_{\rm tan}}+\pi^2 \kappa  \Big) \d \sigma(z).
			\end{aligned}            
		\end{equation}
		Observe that, in the first equality above, the linear terms in $t$ vanish after integration on the symmetric interval $(-h_0/2,h_0/2)$.
		
		By collecting \eqref{eq:a}, \eqref{eq:b} and \eqref{eq:c}, we now conclude the proof of \eqref{eq:Gliminf} by means of \eqref{eq:B} and recalling the expressions \eqref{eq:envK}, \eqref{eq:envKlin} and \eqref{eq:enapi}.

		\medskip
		
		\noindent\underline{\textbf{Limsup inequality.}} We now move to the proof of the limsup inequality \eqref{eq:Glimsup}. We employ the same construction for a recovery sequence presented in \cite{LMP10}, see also \cite[Theorem 6.3]{Lew13}, highlighting how to deal with the pressure term. We thus refer to them for details.
		
		Fix a pair $(V,B)\in  \mathcal{V}(S)\times \mathcal{B}(S)$. By Definition \ref{def2}, $ B $ can be represented as the $ L^2$-limit of symmetric surface gradients $ \operatorname{sym} \nabla^S w^h $, for a certain family of vector fields $ w^h \in W^{1,2}(S;\mathbb{R}^3)$. Without loss of generality, we may assume that the fields $w^h$ are smooth and satisfy
		\begin{equation}
			\label{eq:approB}
			\lim\limits_{h\to 0}h^{1/2} \| w^h \|_{W^{2, \infty}(S)}= 0. 
		\end{equation}
		
		Moreover, $ V \in \mathcal{V}(S) $ can be approximated by a sequence of functions $ v^h \in W^{2,\infty}(S;\mathbb{R}^3)$ such that
		\begin{equation}
			\label{eq:approV}
			\begin{aligned}
				&h^{\alpha-1} \|v^h\|_{W^{2,\infty}(S)} \leq \varepsilon,\\
				&\lim\limits_{h\to 0}\|v^h - V\|_{W^{2,2}(S)}= 0,  \\
				&\lim\limits_{h\to 0}\frac{1}{h^{2\alpha-2}} \left| \{z \in S; \ v^h(z) \neq V(z)\} \right| = 0,
			\end{aligned}
		\end{equation}
		where $ \varepsilon >0 $ is a fixed parameter which will be sent to $0$ afterwards. We also recall, see (6.8) in \cite{Lew13}, that there holds
		\begin{equation}\label{eq:symvanish}
			\lim\limits_{h\to 0}\frac{1}{h}\| \operatorname{sym} \nabla^S v^h \|_{L^2(S)} = 0.
		\end{equation}
		
		The recovery sequence of rescaled deformations is then defined by
		\begin{equation} \label{eq:recover}
			\begin{aligned}
				y^h(z + t\mathbf{n}(z)) &:= z + h^{\alpha-1} v^h(z) + h^{\alpha} w^h(z) \\
				&\quad + t \frac{h}{h_0} \mathbf{n}(z) + t \frac{h^{\alpha}}{h_0} \left( \Pi v_{{\rm tan}}^h - \nabla^S \langle v^h, \mathbf{n} \rangle \right)(z)  \\
				&\quad - t \frac{h^{\alpha+1}}{h_0} \left( \mathbf{n}^T \nabla^S w^h \right)(z) + t \frac{h^{\alpha+1}}{h_0} d^{0,h}(z) + t^2 \frac{h^{\alpha+1}}{2h_0^2} d^{1,h}(z). 
			\end{aligned}
		\end{equation}
		Above, we used the notation
		\[
		v_{{\rm tan}}^h(z) = v^h(z)  - \langle v^h(z), \mathbf{n}(z) \rangle \mathbf{n}(z),
		\]
		while the sequences of vector fields $ d^{0,h}, d^{1,h} \in W^{1,\infty}(S; \mathbb{R}^3)$ satisfy
		\begin{align*}
			&d^{0,h} \xrightarrow[h\to 0]{} d^0  + \frac{\delta}{2} A^2 \mathbf{n} + \frac{\delta}{2} |{A} \mathbf{n}|^2 \mathbf{n},&& \text{strongly in $L^2(S;\mathbb{R}^3),$}\\
			&
			d^{1,h} \xrightarrow[h\to 0]{}  d^1 + \left( \mathbf{n}^T A \Pi - \mathbf{n}^T \nabla^S ({A} \mathbf{n}) \right),&&\text{strongly in $L^2(S;\mathbb{R}^3),$}
		\end{align*}
		where $d^0$ and $d^1$ are specified below in \eqref{eq:d0d1}, while $\delta = 1$ if $\alpha =2$ and $\delta =0$ if $\alpha > 2$. We may additionally require that
		\begin{equation}
			\label{eq:add}
			\lim\limits_{h\to 0}h^{1/2} \left( \|d^{0,h}\|_{W^{1,\infty}(S)} + \|d^{1,h}\|_{W^{1,\infty}(S)} \right) = 0.
		\end{equation}          
		We also choose
		$$
		R_h = \bar{R}_h := \operatorname{Id}_3, \qquad c_h := 0,
		$$
		so that $\widetilde{y}^h = y^h$. 
		
		As in Step 2 of \cite[Proof of Theorem 6.3]{Lew13}, \eqref{eq:convuva} is now a consequence of \eqref{eq:recover}, \eqref{eq:approB}, \eqref{eq:approV} and \eqref{eq:add}. Furthermore, a direct computation yields
		\begin{align*}
			&V^h_\alpha[y^h](z)  = v^h(z) + hw^h(z) + \frac{1}{24} h^2 d^{1,h}(z), \\
			&\frac{1}{h} \operatorname{sym} \nabla^S V^h_\alpha[y^h](z) = \frac{1}{h} \operatorname{sym} \nabla^S v^h(z) + \operatorname{sym} \nabla^S w^h + \frac{h}{24} \operatorname{sym} \nabla^S d^{1,h}(z),
		\end{align*}
		whence \eqref{eq:convuvb} and \eqref{eq:convuvc} follow by \eqref{eq:approB}, \eqref{eq:approV}, \eqref{eq:symvanish} and \eqref{eq:add}. 
		
		Next, we set
		\begin{equation}
			\label{eq:Kh}
			K^h(x) := \frac{(\nabla_h y^h(x))^T \nabla_h y^h(x) - \operatorname{Id}_3}{h^\alpha},
		\end{equation}
		which, by (6.13) and (6.17) in \cite{Lew13}, satisfies
		\begin{equation}
			\label{eq:boundedKh}
			h^\alpha \|  K^h \|_{L^\infty(S^{h_0})} \leq C\varepsilon,
		\end{equation}
		and
		\begin{equation}
			\label{eq:kh4}
			\lim\limits_{h\to 0}h^{2\alpha} \int_{S^{h_0}} |K^h|^4 \d x = 0.
		\end{equation}
		Moreover, by arguing as in Steps 4 and 5 of \cite[Proof of Theorem 6.3]{Lew13} (see also (6.15) therein), one obtains
		\begin{equation}
			\label{eq:convergeKh}
			\frac{1}{2}K^h \xrightarrow[h\to 0]{} \overline{K_1} \circ \varrho^S + \frac{t}{h_0}\overline{K_2} \circ \varrho^S+\sym\zeta \otimes \mathbf{n},\qquad\text{strongly in } L^2(S^{h_0};\mathbb{R}^{3\times3}),
		\end{equation}
		where
		\begin{equation}
			\label{limits-identify}
			\begin{aligned}
				&K_1(z) = {B(z)-\frac{\delta}{2}(A^2)_{{\rm tan}}(z)},\\
				&K_2(z) = \operatorname{sym}{\left(\nabla^S(A\mathbf{n})-A\Pi\right)_{{\rm tan}}(z)},\\
				&\zeta(z+t\mathbf{n}(z)) = d^0(z) +\frac{t}{h_0}d^1(z).
			\end{aligned}
		\end{equation}        
		
		Now, as shown in Step 6 of  \cite[Proof of Theorem 6.3]{Lew13}, by using \eqref{eq:Kh}, \eqref{eq:boundedKh}, \eqref{eq:kh4}, and \eqref{eq:convergeKh} one deduces
		\begin{equation}\label{eq:lim1}
			\begin{aligned}
				&\limsup\limits_{h\to 0}\frac{1}{h^{2\alpha}}\int_{S}\dashint_{-h_0/2}^{h_0/2} W(\nabla_h y^h(z+t\mathbf{n}(z))) \mathcal{D}_h(z,t)\d t \d \sigma(z)\\
				\le&\limsup\limits_{h\to 0}\frac 12\int_{S}\dashint_{-h_0/2}^{h_0/2} Q_3\left(\frac 12 K^h(z+t\mathbf n(z))\right)\mathcal D_h(z,t)\d t \d \sigma(z)+o_\varepsilon(1)\\
				=& \frac 12 \int_{S}\dashint_{-h_0/2}^{h_0/2} Q_3\Big(\overline{K_1}(z)+d^0(z)\otimes\mathbf n(z)+\frac{t}{h_0}\Big( \overline{K_2}(z)+{d^1(z)}\otimes\mathbf n(z)\Big)\Big) \d t \d \sigma(z) + o_\varepsilon(1) \\
				=&\frac{1}{2} \int_S Q_3(\overline{K_1}(z)+d^0(z)\otimes\mathbf n(z))\d \sigma(z)+\frac{1}{24} \int_S Q_3(\overline{K_2}(z)+{d^1(z)}\otimes\mathbf n(z)) \d \sigma(z) + o_\varepsilon(1) ,
			\end{aligned}
		\end{equation}
		where $o_\eps(1)$ is a quantity which vanishes as $\eps\to 0$.
		
		We now focus on the pressure term. We first observe that, by polar decomposition, we can write
		\begin{equation*}
			\nabla_h y^h(x)=U^h(x)\Big((\nabla_h y^h(x))^T\nabla_h y^h(x)\Big)^\frac 12=U^h(x)(\operatorname{Id}_3+h^\alpha K^h(x))^\frac 12,
		\end{equation*}
		where $U^h(x)$ is an orthogonal matrix. In particular, recalling \eqref{eq:sviluppodet}, there holds
		\begin{align*}
			\det \nabla_h y^h(x)&=\det(\operatorname{Id}_3+h^\alpha K^h(x))^\frac 12=\sqrt{\det(\operatorname{Id}_3+h^\alpha K^h(x))}\\
			&=\sqrt{1+h^\alpha\tr K^h(x)+\iota_2(h^\alpha K^h(x))+h^{3\alpha}\det K^h(x)}\\
			&\le1+h^\alpha\tr \frac 12 K^h(x)+Ch^{2\alpha}|K^h(x)|^2 +Ch^{3\alpha}|K^h(x)|^3 , 
		\end{align*}
		where we exploited the elementary inequality $\sqrt{1+\lambda}\le 1+\frac 12\lambda$, for $\lambda\ge -1$. By means of the above estimate, together with \eqref{eq:boundedKh} and \eqref{eq:convergeKh}, we then infer
		\begin{equation}\label{eq:equal1}
			\begin{aligned}
				&\limsup\limits_{h\to 0}\frac{1}{h^\alpha}\int_{S}\dashint_{-h_0/2}^{h_0/2} (\det\nabla_h y^h(z+t\mathbf{n}(z))-1) \mathcal{D}_h (z,t)\d t \d \sigma(z) \\
				\leq& \lim\limits_{h\to0}\left(\int_{S}\dashint_{-h_0/2}^{h_0/2} \tr \left(\frac{1}{2}K^h(z+t\mathbf{n}(z))\right) \mathcal{D}_h (z,t)\d t \d \sigma(z) +Ch^\alpha\|K^h\|_2^2+C\varepsilon h^\alpha\|K^h\|_2^2\right)\\
				=& \int_{S}\dashint_{-h_0/2}^{h_0/2} \tr \Big(\overline{K_1}(z)+d^0(z)\otimes\mathbf n(z)+\frac{t}{h_0}\Big( \overline{K_2}(z)+{d^1(z)}\otimes\mathbf n(z)\Big)\Big)  \d t \d \sigma(z)\\
				=& \int_{S} \tr \overline{K_1}(z)+\langle d^0(z),\mathbf n(z)  \rangle \d \sigma(z).
			\end{aligned}
		\end{equation}
		By combining \eqref{eq:lim1} and \eqref{eq:equal1}, we now obtain
		\begin{align*}
			\limsup\limits_{h\to 0}\frac{1}{h^{2\alpha}}\overline{\mc E}^{\alpha,\pi}_h(y_h)\le& \frac 12 \int_S\Big( Q_3(\overline{K_1}+d^0\otimes \mathbf n)+2\pi\langle d^0,\mathbf n\rangle\Big) \d\sigma(z)+\pi\int_S \tr \overline{K_1} \d\sigma(z)\\
			&+\frac{1}{24}\int_S Q_3(\overline{K_2}+{d^1}\otimes\mathbf n) \d \sigma(z)+o_\varepsilon(1).
		\end{align*}
		We finally choose $d^0$ and $d^1$ so that (recall \eqref{eq:quad2'} and \eqref{Q2pi})
		\begin{subequations}\label{eq:d0d1}
			\begin{align}
				Q_2^\pi \left(z,K_1(z)\right)&= Q_3\left(\overline{K_1}(z)+d^0 (z) \otimes \mathbf{n}(z)\right){+}2\pi  \langle d^0 (z), \mathbf{n}(z)\rangle,\\		Q_2(z,K_2(z))&=Q_3(\overline{K_2}(z)+ d^1(z)\otimes \mathbf{n}(z)),
			\end{align}          
		\end{subequations}
		whence by \eqref{Q2pivalue} we deduce
		\begin{align*}
			\limsup\limits_{h\to 0}\frac{1}{h^{2\alpha}}\overline{\mc E}^{\alpha,\pi}_h(y_h)\le& \frac 12 \int_SQ_2(z,K_1(z)) \d\sigma(z)+\frac{1}{24}\int_S Q_2(z,K_2(z)) \d \sigma(z)\\
			&+\pi\int_S\left(\frac 12 \mathfrak L(z,K_1(z))+ \tr_z {K_1}(z)\right) \d\sigma(z)+\frac{\pi^2}{2}\int_S\kappa(z) \d\sigma(z)+o_\varepsilon(1).
		\end{align*}
		We thus conclude by means of \eqref{limits-identify}, recalling expressions \eqref{eq:envK}, \eqref{eq:envKlin}, and \eqref{eq:enapi}, since one can deal with the term $o_\varepsilon(1)$ by considering a recovery sequence of the form \eqref{eq:recover} obtained via a diagonal argument and by sending $\eps\to 0$.
	\end{proof}

	\bigskip
	
	\noindent\textbf{Data availability statement.} No data were generated or analysed in this study.
	
	\bigskip
	
	\noindent\textbf{Conflict of interest statement.} The authors declare that they have no competing interests relevant to the content of this article.
	\bigskip
	
	\noindent\textbf{Acknowledgements and funding.}
	The authors have been partially supported by the Ministry of Education, Youth, and Sports of the Czech Republic via the project Ferroic Multifunctionalities (FerrMion), no. CZ.02.01.01/00/22\_008/0004591, co-funded by the European Union, and by the grant PRIMUS/24/SCI/020. F. R. also acknowledges the support of the Czech Science Foundation (GA\v CR) via the project 26-21585K. F. R. is member of GNAMPA (INdAM).

	\newcommand \auth{} 
	\newcommand \jour {} 
	\newcommand \book {}

\end{document}